\documentclass[11pt]{amsart}

\makeatletter
\@namedef{subjclassname@2020}{\textup{2020} Mathematics Subject Classification}
\makeatother

\usepackage{fullpage}
\usepackage[all]{xy}
\usepackage{amssymb,latexsym}
\usepackage{graphicx}
\usepackage{amsmath,amsthm}
\usepackage{enumerate}
\usepackage{url}

\usepackage{hyperref}
\hypersetup{
	colorlinks   = true, 
	urlcolor     = blue, 
	linkcolor    = blue, 
	citecolor   = red 
}

\usepackage[msc-links]{amsrefs}

\usepackage[usenames,dvipsnames]{color}

\newcommand{\y}{{{\bf y}}}

\newcommand{\kk}{{{\bar k}}}

\newcommand{\Fb}{{\bf F}}
\newcommand{\Gb}{{\bf G}}

\newcommand{\Mb}{{\bf M}}

\newcommand{\Cb}{{\bf C}}

\newcommand{\Z}{{\mathbb Z}}

\newcommand{\R}{{\mathbb R}}
\newcommand{\Q}{{\mathbb Q}}
\newcommand{\Pp}{{\mathbb P}}
\newcommand{\CC}{{\mathbb C}}

\newcommand{\A}{{\mathcal A}}
\newcommand{\B}{{\mathcal B}}
\newcommand{\Cc}{{\mathcal C}}
\newcommand{\D}{{\mathcal D}}
\newcommand{\Ee}{{\mathcal E}}

\newcommand{\Ll}{{\mathcal L}}

\newcommand{\PP}{{\mathcal P}}
\newcommand{\OO}{{\mathcal O}}

\newcommand{\Xc}{{\mathcal X}}

\newcommand{\p}{{\mathfrak p}}

\newcommand{\fd}{{\mathfrak d}}

\newcommand{\op}{{\rm ord_\p}}

\newtheorem{thm}{Theorem}[section]

\newtheorem{lema}[thm]{Lemma}

\newtheorem{ques}[thm]{Question}
\newtheorem{conj}[thm]{Conjecture}
\newtheorem{prop}[thm]{Proposition}

\numberwithin{equation}{section}
\begin{document}

	\title[On powerful values of polynomials ]{On powerful values of polynomials   over  number fields}
	
	\subjclass[2020]{Primary 	11R09; Secondary 11J97.}
	\keywords{Powerful values of polynomials, Vojta's conjecture on bounded degree algebraic numbers}
	\author{Sajad Salami}
	\address{Institute of Mathematics,  Statistics, State University of Rio de Janeiro,  Rio de Janeiro, Brazil}
	\email{sajad.salami@ime.uerj.br}
\date{\today}

\maketitle
\begin{abstract}
	
	Let  $\B=\{b_i \}_{i=1}^\infty$ be a fixed  sequence  of  pairwise distinct elements of  a  number field $k$.
Given the integers $2\leq s \leq r$, 
assuming  a quantitative version of  Vojta's conjecture  on  the  bounded degree algebraic numbers on a number field $k$, we provide lower and upper bounds  for the cardinal  number  of  $\Gb_{r,s}^{\B_M}$ the set of
 polynomials $f\in k[x]$ of  degree $r\geq 2$ whose irreducible  factors have multiplicity strictly less than $s$  and 
$f(b_1),\cdots,  f(b_M)$ are nonzero $s$-powerful elements in $k$, where  $M=2r^2+6r +1$ if $r=s$, and $2sr^2+ s r+1$ otherwise.
Moreover, considering certain conditions on $\B$, we show  the existence of an integer $M_0> M$ such that no polynomial in $\Gb_{r,s}^{\B_M}$  takes  $s$-powerful values at all  of $b_1, \cdots, b_n $ for  $n\geq M_0$.
\end{abstract}

\section{Introduction and main results}\label{Intm2}
Let $k$ be a number field and denote by  $\OO_k$  the ring of integers in $k$. 
Given integer  $ s\geq 2$,  an  element $\alpha$ of  $\OO_k$ 
 is called $s$-{\it powerful} if  for each prime ideal $\p $ of $\OO_k$ dividing  the principle ideal
 $(\alpha)$ we have  $\op(\alpha)\geq s$.
 For convenience, we suppose that $0$ is an $s$-powerful element for any $s\geq 2$.
This definition immediately extends to  elements of $k$. Clearly, any $s$-power in $k$  is an $s$-powerful element.
Similarly, given a polynomial $f \in  k[x]$ of degree $r\geq s $, we say that $f$ is an {\it $s$-powerful polynomial} 
 if each irreducible factor of $f$ has multiplicity at least $s$. It is clear that any $s$-power  polynomial 
 in $k[x]$ is $s$-powerful.

 The powerful values of polynomials have been studied by several authors in the literature,
  \cite{Davenport}, \cite{Riben1974},  \cite{Schinzel},   \cite{Walsh1999}. In the recent work 
	\cite{Pasten2013}, H. Pasten  considered the problem  for  number fields as well as  function fields.
Indeed,  he showed that a certain conjecture by Vojta on Diophantine approximation for number
field extensions of bounded degree, i.e.,   Conjecture \ref{Vojta2}(i) in Sec. \ref{VCAP}, 
 implies  the finiteness of the set of all monic polynomial $f \in k[x]$ of degree $r\geq 2$ whose irreducible factors have multiplicity strictly less than $s\leq r$ and  all $f(b_1), \cdots f(b_M)$ are $s$-powerful for   mutually distinct fixed elements $b_i \in k$, where  $k$ is a fixed number field and  
 	${\bar M}=2r^2+9r+1,$ if $r=s$; and ${\bar M}=2sr^2 + (2s+1)r+1$ otherwise, see \cite[Thm. 2.1]{Pasten2013}.  
 As a corollary, he concluded the existence of   positive constants $M_0$ such that if  $f(1), \cdots, f(M_0)$ are $s$-powerful for some polynomial $f\in k[x]$, then  it has a factor with multiplicity  at least $s$, see \cite[Cor. 2.2]{Pasten2013}.
 He also tried to  drive consequences in logic related to Hilbert’s tenth problem, such as the  B\"uchi's $n$-square problem.

 In this paper, we assume the  quantitative  version of the Vojta's  conjecture  on algebraic points of bounded degree over number field, see Conjecture  \ref{Vojta2}(ii)  in Section \ref{VCAP}.
We fix an arbitrary sequence  $\B=\{ b_1, b_2, \cdots\}$   of pairwise distinct elements in $k$.
Given integers $2 \leq s \leq r < n$, we let $\Fb_{r,s}^{\B_n}$ to  be the set of all monic  polynomials $f\in k[x]$ of degree $r$  such that  $f(b_i)$ is a nonzero  $s$-powerful element in $k$ for each $b_i \in \B_n$, where 
 $\B_n=\{b_1, \cdots, b_n\} \subset \B$.
	Denote by $\Gb_{r,s}^{\B_n}$ the subset of  $\Fb_{r,s}^{\B_n}$ containing polynomials for which all irreducible factors have multiplicity strictly less than $s$. 
It is clear that   $\Fb_{r,s}^{\B_{n+1}} \subseteq \Fb_{r,s}^{\B_n}$  and hence $\Gb_{r,s}^{\B_{n+1}} \subseteq \Gb_{r,s}^{\B_n}$ for all $n\geq 1$. 
 One can think about $\Gb_{r,s}^{\B_n}$ as the set of exceptions to the rule:
{\it if $f$ takes powerful values too many times then $f$ has factors with  exponents $\geq s$.}

The following theorem is the main result of this paper.
\begin{thm}
\label{mainpol1}
Assume   the   Vojta's Conjecture \ref{Vojta2}(ii). 
Given integers $2\leq s \leq r$,  let
$M:=2r^2+6r +1$ if $r=s$, and $2sr^2+ s r+1$ otherwise.  
Then  $\Gb_{r,s}^{\B_M}$ is a finite set.
Moreover,  there exist positive constants $C_0$ and $C_1$,  depending  on $r$, $s$, fundamental quantities of $k$  and  elements of $\B_M$, 
but  it is independent of the polynomial $f \in \Gb_{r,s}^{\B_M}$, 
such that  
   $$ C_0 \leq \# \Gb_{r,s}^{\B_M} \leq C_1.$$
\end{thm}

We remark  that the number $M$ given by this theorem 
    depends only on  the integers $r$ an $s$, but it is independent
 of the sequence $\B$ and its subset $\B_M.$
  In contrast, the proof of Theorem \ref{mainpol1}, which is similar to  \cite[Thm. 2.1]{Pasten2013},  shows that  
$C_0$ and $C_1$    depend on $b_i \in \B_M$, the integers $r$,  $s$, and  basic quantities of the number field  $k$.
We remark that  the number $M$ in  Theorem \ref{mainpol1}  is
smaller than  one given  in \cite[Thm. 2.1]{Pasten2013}. For,  we  used a  different estimation of the height of logarithmic  discriminate of  $k$.
Our result also provides a lower and an upper bound for the set under study, which is not provided in \cite{Pasten2013}.

Now,
let us   consider   sequences  $\Cc  =\{c_i\}_{i=1}^{\infty}, $  $\D =\{d_i\}_{i=1}^{\infty},$ and 
 $\Ee =\{e_i \}_{i=1}^{\infty},$ associated  to the given fixed sequence $\B=\{b_i\}_{i=1}^{\infty}$,  which are defined by
$$c_i:=b_{i+1}-b_i, \ d_i:=b_i/b_{i+1}, \ e_i:=1/b_{i+1}-1/b_i, \ ( i=1,2, \cdots ).$$
As a consequence of the   Theorem \ref{mainpol1},     following  result  shows that
 if we restrict our attention to the    sequences  $\Cc$, $\D$ and $\Ee$ in Theorem \eqref{mainpol1},
 then we can get rid of the set $\Gb_{r,s}^{\B_n} $ for a large enough  $n$.
\begin{thm}	
\label{cor1}
Assume   the part (ii)  of Vojta's Conjecture \eqref{Vojta2}. 
Let $M$ denote the integer  given by  Theorem \ref{mainpol1}. 
If any of  sequences $\Cc, \D,$ and $\Ee$ is  periodic    with period  $m\geq 1$, then 
 there exists an  integer $M_0>0$, depending on $r$, $s$, fundamental quantities of  $k$ and elements of $\B_M$, such that  
$\Gb_{r,s}^{\B_n} = \emptyset$  for $n> M_0$.
\end{thm}

Since all the   sequences $\Cc$, $\D$   and $\Ee$ are periodic with period $m=1$ for the sequence $\B=\{1,2,\cdots\}$, so   Theorem  \ref{cor1} implies  \cite[Cor. 2.2]{Pasten2013}.

The structure of this paper is as follows. In Sec. \ref{anbdh}, we give the preliminaries on  height functions over number fields and
 also the main result of \cite{Ih2008} that will use in  proof of \ref{mainpol1}.
In Sec. \ref{VCAP},  we provide some definitions and terminologies of   Diophantine approximation  to ba able
for  stating   the equivalent versions  of  Vojta's conjecture 
  on bounded degree algebraic points over number fields. 
The section  contains   proofs of  \ref{mainpol1} and  \ref{cor1}.
 
\section{Preliminaries on heights functions}
\label{anbdh}

Given a number field $k$ with   algebraic closure   $\kk$,   we let  $\PP_k$ denote the set of places of $k$ that 
 splits into two disjoint subsets. One, $\PP_k^0$ the set of the finite places, \emph{i.e.},
those corresponding to prime ideals $\p$ of $\OO_k$,
and another one of the infinite places denoted by $\PP_k^\infty$, \emph{i.e.},
those corresponding to real embedding $\sigma : k \hookrightarrow \R $, called the real infinite places, 
union with those corresponding to  pair of conjugate embedding $\sigma , \bar{\sigma}: k \hookrightarrow \CC$ 
that are  called the complex infinite places.
 For any $v \in \PP_k$, denote by $\| \cdot \|_v$  its  associated  almost absolute value defined by 
\[\| \alpha \|_v  := 
\begin{cases}
0 & \text{if  $\alpha=0$ },\\
\# (\frac{\OO_k}{\p_v})^{\op_v (\alpha)} & \text{if $v \in \PP_k^0$ corresponds to a prime $ \p_v \in Spec(\OO_k)$},\\
|\sigma(\alpha)| &  \text{if $v \in \PP_k^\infty $ is  a real infinite place},\\
|\sigma(\alpha)|^2 &  \text{if $v \in \PP_k^\infty $ is a complex infinite place}.
\end{cases}
\]

 Given $P=[\alpha_0: \cdots : \alpha_n] \in \Pp_k^n$, the {\it   multiplicative} and {\it logarithmic} heights  
 are defined  by 
$$H_k(P):= \prod_{v \in \PP_k} \max \{\| \alpha_0 \|_v,  \cdots, \| \alpha_n \|_v\},  $$
$$h_k(P):= \log H_k(P)= \sum_{v \in \PP_k}^{} \log \max \{\| \alpha_0 \|_v,  \cdots, \| \alpha_n \|_v \}.$$
Then,  the  {\it   multiplicative} and {\it logarithmic  heights}  of any $\alpha \in k$  are defined by 
 $H_k(\alpha)=H_k([1:\alpha])$ and $h_k(\alpha)h_k([1:\alpha])$,  where we identified $k$ with affine space in $\Pp_k^1$.
For any finite extension $K|k$,  $\alpha \in k$, and $P \in \Pp_k^n$, one has
$$H_k(\alpha)= H_K(\alpha)^{1/[K:k]}, \  h_k(\alpha)=\frac{1}{[K:k]} h_K(\alpha), $$
$$H_k(P)= H_K(P)^{1/[K:k]}, \  h_k(P)=\frac{1}{[K:k]} h_K(P).$$
Considering these facts, one may extend the definition of height function  to $\Pp_\kk^n$.
In this case, they are called the  {\it absolute multiplicative and additive Weil heights}  and denoted by 
$H(P)$ and $h(P)$, respectively. We note that the action of Galois group of $\Pp_\kk^n$ leaves the absolute  heights invariant.
Moreover, for each $\alpha, \beta \in \kk^*$ and $n \in \Z$, one has the followings:
\begin{equation}
\label{hhha}
h(\alpha ^{n}) = |n|h(\alpha), \ h(\alpha \beta) \leq h(\alpha)+ h(\beta), \  h(\alpha +\beta)\leq h(\alpha) + h(\beta) + \log 2.
\end{equation}
For any polynomial  $ f(x)= a_0  +a_1x + \cdots + a_d x^d \in k[x],$ the {\it absolute multiplicative} and {\it additive heights} are  defined by  
$$H(f):= H([a_0: a_1: \cdots : a_d]), \   \  h(f):= h([a_0: a_1: \cdots : a_d].$$  


Giving  lower and upper bounds for the cardinal number of  the set of  bounded degree algebraic points on a projective line is started by  \cite{Schmidt1993} and continued by \cite{Schmidt1995},  \cite{Masser2007}, \cite{Ih2008}, and so on.
Let us  recall the main result in  \cite{Ih2008} that we will use in the proof  of \ref{mainpol1}.

 Denote  by $N(\Pp_\kk^1; r; T)$ the number of points  $\alpha \in \Pp_\kk^1$ of degree at
 most $r$ and  $h(\alpha)\leq T$ for every constant $T>0$ and integer $r\geq 2$.
Let $Cl_k$ be the class number of $k$,  $\text{Reg}_k$ the regulator of $\OO_k^*$, $w_k$ the number of roots of unity
in $k$, $\zeta_k(s)$ the Dedekind zeta-function of $k$, $\fd_k$ the absolute discriminant of
 $k$, $m_1$ the number of real embedding of $k$,  $m_2$ the number of pairs of complex embedding of $k$,
 and $m=m_1+2m_2$ is the degree of $k$ over $\Q$. For more details on these quantities, we refer the reader to \cite{Neukirch1999}.
Define 
\begin{equation}
\label{shan-const}
a_{k,r}:= \frac{Cl_k \cdot \text{Reg}_k}{w_k \zeta_k(r+1)} \cdot 
\big(\frac{2^{m_1}(2\pi)^{m_2}}{\fd_k^{1/2}}\big)^{r+1} \cdot (r+1)^{m_1+m_2-1},
\end{equation}
and denote $b_{k,r}:= r\cdot a_{k,r} \cdot T ^{mr(r+1)}$ and $T_1=T^{mr(r+1)-r}$.

\begin{thm} 
\label{SuIH} Notation being as above,  for each    $\varepsilon >0 $ one has 
$$b_{k,r}\cdot 2^{-mr(r+1)} T^{mr(r+1)} - O_\varepsilon (T_1\cdot T^\varepsilon) \leq N(\Pp_\kk^1; r; T) 
\leq b_{k,r} \cdot  2^{mr(r+1)}  + O(T_1).$$
In particular,  
$$2^{-mr(r+1)} + o(1) \leq  \frac{N(\Pp_\kk^1; r; T) }{b_{k,r}} \leq  2^{mr(r+1)} + o(1)  \ \ \textit{as} \ \   T \rightarrow \infty.$$
\end{thm}

Without loos of generality, we may suppose that  $k \subset \CC$ and 
\begin{equation} 
\label{polinc}
f(x)=a_0 +a_1x+ \cdots + a_d x^d =a_d \prod_{j=1}^d (x-\alpha_j) \in \CC[x].
\end{equation}
In this case, the {\it Mahler measure} of any $f\in \CC[x]$  is  defined by
$$M(f):= |a_d| \cdot \prod_{j=1}^{d} \max \{ 1, |\alpha_j|\},  $$
where $|\cdot |$ is the usual absolute value on $\CC$.
For $\alpha \in \kk=\CC$,   its Mahler measure is given  by
 $M(\alpha)=M(f_\alpha)$  where  $f_\alpha \in k[x]$ denotes its minimal polynomial.

 The {\it  logarithmic discriminant } of $k$ is  defined by  $d_k := \log \fd_k /[k:\Q].$ 
For  a  tower of number fields $\Q \subseteq k \subseteq K \subset \kk $ with absolute discriminant $\fd_k$ and $\fd_K$, respectively,
the {\it  relative logarithmic discriminant} of $K|k$ is   
$$d_k(K):= \frac{1}{[K:k]} \log \fd_{K/k}- \log \fd_k,$$
where   $\fd_{K/k}$ is the {\it relative discriminant } of the extension $K|k$.
The {\it relative logarithmic discriminant} of each  $\alpha \in \kk $
 is defined by $d_k(\alpha):=d_k(k(\alpha)).$ 
 The following proposition   gives  an upper bound for the logarithmic discriminant $d_k(\alpha)$ that 
we will use in the proof of \ref{mainpol1}. For a proof, we cite  to  \cite{Mahler1964, BERCZES}.
\begin{prop} 
\label{discupb}

Let $k \subset \CC$ be a number field and
$f\in k[x]$ be a polynomial of degree  $d\geq 2$ 
of the form  
\[ 
 	\label{polinc}
 	f(x)=a_0 +a_1x+ \cdots + a_d x^d =a_d \prod_{j=1}^d (x-\alpha_j) \in \CC[x].
  \] 
Define $A(d)=d \log d$ if $k=\Q$, and  $A(d)=(2d-1)\log d$ otherwise. Then,
\begin{itemize}
\item [(i)]  $D(f) = a_d^{2d-2} \prod_{i>j}(\alpha_i -\alpha_j)^2, $  and  $  |D(f)| \leq d^d \cdot M(f)^{2d-2};$
\item [(ii)] If $D(f) \not = 0$, then   $h(D(f)) \leq  2(d-1)h(f) + A(d);$
\item [(iii)] If  $\alpha \in \kk$ is of degree $d \geq 2$,  then
$d_k(\alpha) \leq  2(d-1)h(\alpha)+ A(d).$
\end{itemize}
\end{prop}
\begin{proof}
	See    \cite[Thm. 1]{Mahler1964} for part (i). 
The part (ii) is consequence of part (i)
 for the case  $k=\Q$, and it is given by   \cite[Lem. 3.7]{BERCZES} when $k \not = \Q$.
The part (iii) is   given by \cite[Prop. 1.6.9]{Bombieri2006} in the case $k=\Q$; and generally it comes from part (ii).
\end{proof}

\section{ Vojta's conjecture on   bounded degree algebraic points}
\label{VCAP}
In  this section, we briefly review  the basic definitions  
and results on   Diophantine approximation  over number fields.
For more details, one can refer  to  \cite{Vojta1987, Vojta2011}.
Then, we state the equivalent versions   of   well-known 
 Vojta's conjecture on   bounded degree algebraic points  over  number fields.

 Given    a finite set  $S \subset \PP_k$ containing $\PP_k^\infty$, and 
the  distinct elements $b, \alpha \in k$,
the   {\it proximity functions} with  respect to  $S$ are defined by 
$$m_S(\alpha):= \sum_{v \in S}^{} \log^+ \| \alpha \|_v, \ \text{and} \
m_S(b,\alpha):= m_S(\frac{1}{\alpha-b}).$$
Similarly, the {\it counting functions} with respect to the set $S$ are  defined by
$$N_S(\alpha):= \sum_{v \not \in S}^{} \log^+ \| \alpha\|_v, \ \text{and} \ N_S(b,\alpha):= N_S(\frac{1}{\alpha-b}).$$
By the properties of logarithm function, for any $\alpha \in k $ one has
\begin{equation}
\label{fmt}
m_S(\alpha) + N_S(\alpha) = \sum_{v  \in \PP_k}^{} \log^+ \| \alpha \|_v = h(\alpha ), 
\end{equation}
which  is an  analogue of first main theorem in classic Value Distribution Theory (or equivalently  Nevanlinna Theory).
The proximity and counting function of any $\alpha \in \kk \backslash k$ are defined as
$$m_{S}(\alpha):= \frac{1}{[K:k]}\cdot m_{T}(\alpha),  \ \text{and} \  N_{S}(\alpha):= \frac{1}{[K:k]}\cdot N_{T}(\alpha),$$
where  $K$ is any finite extension of $k$ containing $k(\alpha)$.
These definitions are independent of the choice of the extension $K.$
For an element $b \in k(\alpha)$  distinct from $\alpha$, one can also define 
\[
\label{porx-count}
m_{S}(b, \alpha):= \frac{1}{[k(\alpha):k]}\cdot m_{T}(b, \alpha),  \ \text{and} \  N_{S}(b, \alpha):= \frac{1}{[k(\alpha):k]}\cdot N_{T}(b, \alpha).
\]
It is easy to see that $h(\alpha)= m_S(\alpha)+N_S(\alpha)$ for all $\alpha \in \kk$.

Here is the  Vojta's conjecture on  algebraic points  of bounded degree over number fields, see \cite{Vojta1987, Vojta2011} for more general version.

\begin{conj} 
	\label{Vojta0} 
	Let $k$ be a number field, $\kk$ its algebraic closure and  $S\subset \PP_k$ a finite set containing  $\PP_k^\infty$.
	Let $b_1, \cdots, b_n$ be pairwise distinct elements of $k$ and $d\geq 2$ an integer. Then, one has the following
	 equivalent statements:
	\begin{itemize}
	\item [(i)] For any $\epsilon >0 $, there exists a constant $c_\epsilon$ depending on $\epsilon$  and previous data, such that
	the inequality
		\[
		\label{vI0}
		\sum_{i=1}^n m_S(b_i,\alpha) \leq (2+ \epsilon ) h(\alpha) + d_k(\alpha) + c_\epsilon,
	\]
	holds for  all $\alpha \in \kk$ with   $[k(\alpha):k] \leq d$ and different from  all $b_i$'s.
	\item [(ii)] For any $\epsilon >0 $, the inequality
	\[
		\label{vI1}
		\sum_{i=1}^n m_S(b_i,\alpha) < (2+ \epsilon ) h(\alpha) + d_k(\alpha),
	\]
	holds for   all but finitely many $\alpha \in \kk$ with   $[k(\alpha):k] \leq d$ and different from   all $b_i$'s. 
	\end{itemize}
\end{conj}

 Using the qualities \ref{porx-count} and   following   proof \ref{fmt} as in \cite{Vojta2011}, 
one can see that  the inequality
\begin{equation}
\label{equi1}
h(\alpha)\leq m_S(b,\alpha)+N_S(b,\alpha) + h(b) + [k(\alpha):\Q] \cdot \log 2,
\end{equation}
holds for any $\alpha \in \kk$ and $b\in k$  distinct from $\alpha$.
Applying this inequality,  the conjecture \ref{Vojta0} can be restated  as follows.

\begin{conj}
\label{Vojta1}
Let $k$ be a number field, $\kk$ its algebraic closure and  $S\subset \PP_k$ a finite set containing  $\PP_k^\infty$.
Let $b_1, \cdots, b_n$ be pairwise distinct elements of $k$ and $d\geq 2$ an integer.  
Then, one has the following
equivalent statements:
\begin{itemize}
	\item [(i)] For any $\epsilon >0 $, there exists a constant $c_\epsilon $ depending on $\epsilon$  and previous data, such that
	the inequality
	\[\label{vI10}
	(n-2-\epsilon ) h(\alpha)  \leq d_k(\alpha) + \sum_{i=1}^n N_S(b_i, \alpha) +c_\epsilon,
	\]
		holds for  all $\alpha \in \kk$ with   $[k(\alpha):k] \leq d$ and different from  all $b_i$'s.
\item [(ii)] For  any $\epsilon >0 $, the inequality
\[
\label{vI11}
(n-2-\epsilon ) h(\alpha)  < d_k(\alpha) + \sum_{i=1}^n N_S(b_i, \alpha),
\]
holds for   all but finitely many $\alpha \in \kk$ with   $[k(\alpha):k] \leq d$ and different from $b_i$'s.
 	\end{itemize}
\end{conj}

The {\it truncated counting function } on $\kk$ is defined by 
$$N_{S}^{(1)}(b, \alpha):= \sum_{w  \in M_K^0}^{}\min \{1,\max \{0, \op_w (\alpha-b)\}\} \cdot \log (\# (\OO_K/\p_w)),,$$
where   $ b \in k$  is distinct from $\alpha \in  \kk$,
$K\supseteq k(\alpha)$  and  $\p_w \subset \OO_K$ is a prime ideal corresponding  to  $w \in M_K^0$ that lies  over some $v \in M_k \backslash S$.
Here is the  truncated version of the  Vojta's conjecture.
\begin{conj}
\label{Vojta2}
Let $k$ be a number field, $\kk$ its algebraic closure and  $S\subset \PP_k$ a finite set containing  $\PP_k^\infty$.
Let $b_1, \cdots, b_n$ be pairwise distinct elements of $k$ and $d\geq 2$ an integer. Then,
one has the following equivalent statements:
\begin{itemize}
	\item [(i)] For any $\epsilon >0 $, there exists a constant $c_\epsilon $ depending on $\epsilon$  and previous data, such that
	the inequality
	\[
	\label{vI20}
	(n-2-\epsilon ) h(\alpha)  \leq d_k(\alpha) + \sum_{i=1}^n N_S^{(1)}(b_i, \alpha) +c_\epsilon,
	\]
		holds for  all $\alpha \in \kk$ with   $[k(\alpha):k] \leq d$ and different from  all $b_i$'s.
	\item [(ii)]  For any $\epsilon >0$, the inequality
\[
\label{vI2}
(n-2-\epsilon ) h(\alpha)  < d_k(\alpha) + \sum_{i=1}^n N_S^{(1)}(b_i, \alpha), 
\]
holds for   all but finitely many  $\alpha \in \kk$ with   $[k(\alpha):k] \leq d$ and different from $b_i$'s. 
	\end{itemize}
\end{conj}

We notice that   conjecture \eqref{Vojta2} is a special case of a Vojta's general conjecture   
 on the bounded degree  algebraic points  on algebraic  varieties, see \cite[Conj. 25.1]{Vojta1998}.
 It is equivalent to the  non-truncated version \eqref{Vojta1}. Indeed, 
since $N_S^{(1)}(b, \alpha) \leq N_S(b, \alpha)$ holds by definitions, so the truncated version 
 \eqref{Vojta2} implies the non-truncated one   \eqref{Vojta1}.
 The converse is the special case of the  theorem (3.1)  in \cite{Vojta1998}, where a more general ABC conjecture is stated. 



It is remarkable that  the finite sets  of elements in $ \kk$ of degree at most $d$
for which  the  inequalities  of part (ii) in   conjectures   \eqref{Vojta0} \eqref{Vojta1} and \eqref{Vojta2}
 do  not hold,  depends on $b_i$'s, $\epsilon$, $c$, $k$, and $d$.  In practice, determining  effectively  this finite set of elements   is very hard task.

\section{Proof of the main results}
\label{mainproof1}
In this section,  we  give the proof of   Theorems \ref{mainpol1} and  \ref{cor1}.
\subsection{Proof of Theorem \ref{mainpol1}}

Suppose  that $2 \leq s \leq r$ are  integers and  $b_1, \cdots, b_M \in k$ are distinct elements, where 
$M=2r^2+6r + 1$ and  $M=2sr^2+sr+1, $ otherwise.
Then,   consider the subset    $\B_M:= \{b_1, b_2, \cdots, b_M\}$ of the sequence $\B= \{b_1, b_2, \cdots \}\subset k.$ 
 Let $f \in \Gb_{r,s}^{\B_{M}}$ with  factorization 
 $f=f_1^{s_1}\cdots f_t^{s_t}$,  where $f_j\in k[x]$ are  monic irreducible  polynomial of degree $d_j:=\deg(f_j)$ and 
 define $s^+ =\max \{ s_1, \cdots, s_t\}$.
For each $j=1, \cdots, t$,  let $\alpha_j \in \kk$ be an arbitrary root of $f_j$, 
 $k_j:=k(\alpha_j)$ and  $g:=f_1 \cdots f_t$, which is  of degree $d:= d_1+\cdots + d_t$. 
Let $S\subset \PP_k$ be a finite  subset of $\PP_k$  that is  the union of the sets $\PP_k^\infty$,
 poles of each  $b_i\in \B_{M}$, and  the places above which two or more   of  $ b_j'$s meet.
 Note that $\alpha_j \neq b_i$ for all $i$  and $j,$ because $f(b_j) \neq 0$. 

By Vojta's  conjecture  \ref{Vojta2}(ii) with   $b_i\in  \B_{M}$, the set $S$,  and integer $r\geq 2$, we conclude that  
  for  any given $\epsilon >0 $ the following inequality
\begin{equation}
\label{veq1}
(M-2-\epsilon ) h(\alpha) <  d_k(\alpha) +\sum_{i=1}^{M} N_S^{(1)}(b_i, \alpha)
\end{equation}
holds for   all but finitely many $\alpha \in \kk$ with   $[k(\alpha):k] \leq r $ and  $\alpha\not = b_i$'s. 
 Let us  denote by $N_{k,r}^{\B_M}$ the set of such elements  $\alpha \in \kk$ for which   \eqref{veq1} does not hold,
  and denote its cardinal number by $n_{k,r}^{\B_M}$. 
Since we are going to estimate $\# \Gb_{r,s}^{\B_{M}}$, so for a while we ignore the polynomials $f \in \Gb_{r,s}^{\B_{M}}$ 
that have some roots in the set $N_{k,r}^{\B_M}$. 
We recall them in the moment of estimating    $\# \Gb_{r,s}^{\B_{M}}$.

Thus, assuming Vojta's  conjecture  \ref{Vojta2}(ii), for any  $\epsilon >0 $
\begin{equation}
\label{veq1a}
(M-2-\epsilon ) h(\alpha_j)  <  d_k(\alpha_j) +\sum_{i=1}^{M} N_S^{(1)}(b_i, \alpha_j),
\end{equation}
where  $\alpha_j$ is  a root of $f_j$ for each $j=1,\cdots, t.$ 
Applying  the part (ii) of  Theorem \ref{discupb} to each of $\alpha_j$'s and using $d_j\leq d$, 
 leads to
\begin{equation}
	\label{veq1aa}
	d_k(\alpha_j)\leq 2(d_j-1)h(\alpha_j)  + A(d_j)\leq  2(d-1)h(\alpha_j)  + A(d),
	\end{equation}
where $A(d)=d \log d$ if $k=\Q$ and $(2d-1) \log d$ otherwise, for any integer $d \leq r$.
Substituting  \eqref{veq1aa} in   \eqref{veq1a}, and using the fact that $A(d)\leq A(r)\leq 2r \log r$ leads to
\begin{equation}
\label{veq1b}
(M-2d-\epsilon ) h(\alpha_j) < \sum_{i=1}^M N_S^{(1)}(b_i, \alpha_j)  +  c_1, 
\end{equation}
where $c_1:=M(B+ r\cdot \log 2)  + 2r \log r.$
Then, multiplying the both side  with $d_j$ and  summing-up, one can   obtain that
\begin{equation}
\label{veq2}
\sum_{j=1}^{t}(M-2d-\epsilon)h_{k_j}(\alpha_j) <   \sum_{j=1}^{t} \sum_{i=1}^{M} d_j N_{S}^{(1)} (b_i, \alpha_j) 
 + rc_1. 
\end{equation}
We are going  to give an upper bound for the term involving truncated function in  \eqref{veq2}. 
But to do this,  we need the following lemma.
\begin{lema}
\label{veq0}
Let  $D(g)$ be the   discriminant  of  polynomial $g=f_1\cdots f_t$ of degree $d\geq 2$ and let
$A(d)$  be as above. Then
$$h(D(g)) \leq 2(d-1) \sum_{j=1}^t h_{k_j}(\alpha_j)  + 4d(d-1) +A(d).$$
\end{lema}
\begin{proof}
We assume that $\alpha_{ji}$ are the roots of $f_j$ for $1 \leq i \leq d_j $.
Since the absolute heights are invariant by  the action of  Galois group of $\Pp_{\kk}^1$, 
so 
\begin{align*}
\sum_{j=1}^t h(f_j) & \leq \sum_{j=1}^t( \sum_{i=1}^{d_j}  h(\alpha_{ji}) + (d_j-1)\log 2)\\
&  \leq \sum_{j=1}^t d_j h(\alpha_j) +   \sum_{j=1}^t (d_j-1)\log 2\\
& \leq  \sum_{j=1}^t h_{k_j}(\alpha_j) + (d-t)\log 2.
\end{align*}
Hence,  using the properties of  heights functions \cite[Prop. B.7.2]{Hindry2000}, we have 
\begin{align*}
h(g)=h(f_1\cdots f_t) &   \leq \sum_{j=1}^t [h(f_j) +(d_j+1)\log(2)] \\
& =\sum_{i=1}^t h(f_j)  + (d + t) \log 2 \leq \sum_{j=1}^t h_{k_j}(\alpha_j) + 2d  \log 2.
\end{align*}
By   Theorem \ref{discupb}(ii), we obtain the desired inequality, 
\[h(D(g))  \leq 2(d-1)h(g)+A(d)\leq 2(d-1) \sum_{j=1}^t h_{k_j}(\alpha_j)  + 4d(d-1)\log 2 +A(d).\]
\end{proof}

Let $\D$ be the reduced  divisor on $Spec (\OO_k)$ whose support consists of the union of $S$, zeros of $D(g)$, and   poles of  $\alpha_j$'s.

\begin{lema}
\label{veq0a}
With notation as above,  we have:
\begin{equation}
	\label{eeqq1}
\sum_{j=1}^{t} \sum_{i=1}^{M}  d_jN_S^{(1)} (b_i, \alpha_j) \leq 
\big [  \frac{Ms^+}{s} +d(2d-1) \big ]\sum_{j=1}^{t}  h_{k_j}(\alpha_j)  +r c_2,
\end{equation}
where  
$$c_2:= M(B+\log 2) + \sum_{\p \in S} \log \# (\OO_k/\p) +2r \log r+ A(r).$$
\end{lema}
\begin{proof}
By  changing the  order of sums in   left-hand side  of \ref{veq2}
 and following the last part of the proof of   \cite[Lemma 4.9]{Pasten2013},  we have
\begin{align*}
\sum_{i=1}^{M} \sum_{j=1}^{t} d_jN_{S}^{(1)} (b_i, \alpha_j) &\leq 
\frac{1}{s} \sum_{i=1}^{M} \sum_{j=1}^{t} s_j d_jh (b_i-\alpha_j) + d \deg(\D)\\
&  \leq \frac{1}{s}  \sum_{i=1}^{M} \big (\sum_{j=1}^{t} s_j d_j[h (b_i) + h(\alpha_j) + \log 2 ]\big)+ d \deg(\D).
\end{align*}
Since $t \leq r=\sum_{j=1}^{t} s_jd_j$,  and $s_j\leq s^+$,  so we have 
\begin{align*}
\sum_{i=1}^{M} \sum_{j=1}^{t} d_jN_{S}^{(1)} (b_i, \alpha_j) 
& \leq \frac{1}{s}  \sum_{i=1}^{M}[ \sum_{j=1}^{t} s_j d_j h(\alpha_j) + r (h(b_i) + \log 2)] + d \deg(\D)\\
& \leq \frac{M}{s} \sum_{j=1}^{t} s_j h_{k_j}(\alpha_j) + \frac{M r (B+ \log 2)}{s}  + d \deg(\D)\\
& \leq \frac{Ms^+}{s} \sum_{j=1}^{t}  h_{k_j}(\alpha_j) + M r (B+ \log 2)  + d \deg(\D).
\end{align*}
To give an upper bound on the $\deg (\D)$ in terms of $h(\alpha_j)$'s, 
we assume that  $S'$ and  $S_j$ are the  subsets of $ \PP_k^0$ such that $D(g)$ vanished at $\p$, 
 $\alpha_j$ has a pole above $\p$, respectively.  We let $S''$ to be the union of $S_j$ for $j=1,\cdots, t$. 
  Then, letting  $a(S) := \sum_{\p \in S} \log \# (\OO_k/\p)$, we have 
\begin{align*}
\deg(\D) &= \sum_{\p \in S''} \log \# (\OO_k/\p)  + \sum_{\p \in S'} \log \# (\OO_k/\p) + a(S) \\
& = \sum_{j=1}^t \sum_{\p \in S_j} \log \# (\OO_k/\p)  +\# S' + a(S)\\
& =  \sum_{j=1}^t h_{k_j}(\alpha_j)  +h(D(g)) +a(S). 
\end{align*}
 Using   \eqref{veq0},   and   $A(d)\leq A(r)$ for $d\leq r$,  we get that 
\begin{align*}
\deg(\D) & \leq  \sum_{j=1}^t h_{k_j}(\alpha_j) + 2(d-1) \sum_{ j=1}^t h_{k_j}(\alpha_j)  + a(S)  + A(r) + 4r(r-1) \\
& \leq  (2d-1 )\sum_{ j=1}^t h_{k_j}(\alpha_j) + a(S)  + A(r) + 4r(r-1).
  \end{align*}
Multiplying the last inequality by $d$, gives that
\[d\deg(\D)  \leq  d(2d-1 )\sum_{ j=1}^t h_{k_j}(\alpha_j) + r[a(S)  + A(r) + 4r(r-1)].\]
Putting  all of the above inequalities together leads to  desired one \eqref{eeqq1}. 
\end{proof}

By    Lemma \ref{veq0a}, one can rewrite  \eqref{veq2} as follows,
\begin{equation}
\label{veq5}
\sum_{j=1}^{t}[M(1-\frac{s^+}{s})- 2d^2 -d- \epsilon]h_{k_j}(\alpha_j) <  r (c_1+c_2). 
\end{equation}
\begin{lema}
For integers $2\leq s \leq r$, let   $M=2r^2 +6r+1$ if $r=s$ and $2sr^2+sr+1$ otherwise.
Then, for each  $1\leq d \leq r$, we have 
\begin{equation}
\label{veq6}
M(1-\frac{s^+}{s})- 2d^2 -d\geq \frac{1}{r}.
\end{equation}
\end{lema}
\begin{proof}
First, we have  $r-s^+\geq  d-1$. 
Indeed,   if  $j_0$ is an index such that $s_{j_0}=s^+$, then 
$$r=\sum_{j=1}^{t}s_j d_j \geq s^+ d_{j_0} + \sum_{j\not = j_0}^{t}d_j \geq s^+ + d_{j_0}-1 + \sum_{j\not = j_0}^{t}d_j =s^+ +d-1.$$
Thus,
\begin{equation}
\label{veq00}
1-\frac{s^+}{s} \geq 
\begin{cases} \frac{d-1}{r} &\mbox{if }  s=r \\ 
\frac{1}{s} & \mbox{otherwise}. 
\end{cases} 
\end{equation}
 In the case  $s=r$, since $M=2r^2+6r+1$ and  $r-s^+ \geq d-1\geq 1$, so
\begin{align*}
M(1-\frac{s^+}{s})- 2d^2 -d & \geq  M(\frac{d-1}{r}) -2d^2 -d\\
 &\geq \frac{d-1}{r}(M -\frac{2rd^2+rd}{d-1} )\\
&  \geq \frac{d-1}{r}(M -2rd -3r - \frac{3r}{d-1} ).
\end{align*}
 The facts $-3r/(d-1)\geq -3r$  and $(d-1)/r \geq 1/r$ for $d-1\geq 1$,    implies that
\begin{align*}
M(1-\frac{s^+}{s})- 2d^2 -d & \geq   \frac{d-1}{r}(M-2rd -6r )\\
&\geq  \frac{1}{r} (M -2r^2 -6r ) \geq   \frac{1}{r}.
\end{align*}
In the case  $s< r$, by   $M=2sr^2+sr+1$  and   $1-s^+/s \geq 1/s$ we have:
\begin{align*}
 M(1-\frac{s^+}{s}) -2d^2 - d &\geq  M/s -2d^2 - d \\
& \geq  \frac{1}{s}(M -2sd^2-sd)  \\
& \geq \frac{1}{r}(M -2sr^2-sr) \geq \frac{1}{r}.
\end{align*}
\end{proof}

Now, using \eqref{veq6}  in either cases,  the inequality \eqref{veq5} can be rewritten  as follows,
\begin{equation}
\label{veq7}
(\frac{1}{r}- \epsilon)\sum_{j=1}^{t}h_{k_j}(\alpha_j) <  r(c_1+c_2).
\end{equation}
Then, taking    $\epsilon: =1/(r+1), $ and $   c_3:=  r^2(r+1)(c_1+c_2),$
   the inequality \eqref{veq7} implies  that 
$$h(\alpha_j) < d_j h(\alpha_j)=h_{k_j}(\alpha_j) \leq \sum_{j=1}^{t}h_{k_j}(\alpha_j) < c_3. $$
We note that the consonant $c_3$ depends only on $r$, $s$, $k$  and $b_1, \cdots, b_M$, 
 but  it is independent of the polynomial $f \in \Gb_{r,s}^{\B_{M}}$. 
Let $N(\Pp_\kk^1; r; c_3)$ be the set of algebraic numbers $\alpha\in \kk$ of degree at most $r$ and height at most $c_3$, and denote by  $n(\Pp_\kk^1; r; c_3)$   its cardinal number.
By the famous Northcott's theorem \cite{Northcott1949},  $N(\Pp_\kk^1; r; c_3)$ is a positive number. Letting 
$c_4:=c_3^{mr(r+1)-r}$ with $m=[k:\Q]$, and  applying  Theorem \ref{SuIH}  assuming $\varepsilon:=1$,   $T:=c_3$  and  $T_1:=c_4$,
 gives us  two constants $c_5, c_6>0$,  depending on $r$, $s$ and $k$ but not on $b_i \in \B_M$,  such that
\[b_{k,r} \cdot 2^{-mr(r+1)}  + c_5 \cdot c_4 \cdot c_3 \leq N(\Pp_\kk^1; r; c_3) \leq 	b_{k,r} \cdot 2^{mr(r+1)}  + c_6 \cdot c_4,\]
where $b_{k,r} = r\cdot a_{k,r} \cdot c_3^{mr(r+1)}$  and  $a_{k,r}$ is given by  \ref{shan-const}.
Let $\A_{k,r}^{\B_M}$ be the union of    $N(\Pp_\kk^1; r; c_3)$ and  $N_{k,r}^{\B_M}$ defined in the beginning of the proof. Then
\[b_{k,r} \cdot 2^{-mr(r+1)}  + c_5 \cdot c_4 \cdot c_3   \leq \# \A_{k,r}^{\B_M}  \leq 	b_{k,r} \cdot 2^{mr(r+1)}  + c_6 \cdot c_4 + n_{k,r}^{\B_M}.\]
Since for each $f \in \Gb_{r,s}^{\B_{M}}$ has at most $r$ distinct roots in $\A_{k,r}^{\B_M}$, so we conclude that
\[b_{k,r} \cdot 2^{-mr(r+1)}  + c_5 \cdot c_4 \cdot c_3   \leq \# \Gb_{r,s}^{\B_{M}}  
\leq 	r \cdot (b_{k,r} \cdot 2^{mr(r+1)}  + c_6 \cdot c_4 + n_{k,r}^{\B_M}).\]
  Therefore, we  obtain the desired lower and upper bounds for $\# \Gb_{r,s}^{\B_{M}}$, i.e.,
 $C_0 \leq \# \Gb_{r,s}^{\B_{M}} \leq C_1$, where
	\[C_0:=b_{k,r} \cdot 2^{-mr(r+1)}  + c_5 \cdot c_4 \cdot c_3 + n_{k,r}^{\B_M}, \]
	and 
\[C_1:= r \cdot (b_{k,r} \cdot 2^{mr(r+1)}  + c_6 \cdot c_4 + n_{k,r}^{\B_M}).\]
This completes the proof of Theorem  \ref{mainpol1}.

\subsection{Proof of Theorem \ref{cor1}.}
\label{corproof}

To prove Theorem \ref{cor1}, we start  with the following  lemma.
\begin{lema}
If any of the sequences $\Cc, \D$ and $\Ee$ is periodic with period $t \geq 1$, then
 for each $\ell=q t+p$ with $q \geq 0$ and $1\leq p \leq t$, we have:
\begin{itemize}
	\item[(i)] $c_\ell=c_p$ implies that  $b_\ell = b_p+ q(b_{t+1}-b_1),$ in particular,\\ $b_{qt}=b_t+(q-1)(b_{t+1}-b_1);$
	\item[(ii)] $d_\ell=d_p$ implies that $b_{\ell+1} =b_{p+1} (b_{t+1}/b_1)^q,$ 
	in particular, $b_{q t}=b_t (b_{t+1}/b_1)^{q-1};$
	\item[(iii)]  $e_\ell=e_p$ implies that $1/b_\ell = 1/b_p+ q(1/b_{t+1}-1/b_1),$ in particular, 
	\\ $1/b_{qt}=1/b_t+(q-1)(1/b_{t+1}-1/b_1).$
\end{itemize}
\end{lema}
\begin{proof}
We prove just the part (i) of this lemma  by induction on $q$, and  leave the other cases to the reader.
We assume that $\ell=t+p$ and  $1 \leq p \leq t$. Then,  we have 
\begin{align*}
b_{t+p}-b_{t+1}& =c_{t+1} + c_{t+2} + \cdots +c_{t+p-1}\\
 & = c_1 + c_2 + \cdots + c_{p-1} =b_p-b_1.
\end{align*}
Hence, the assertion is true in the case $q=1$, i.e.,  $b_{t+p}-b_p=b_{t+1}-b_1$.
We assume that it is true for  $\ell=qt+p$ with $1\leq p \leq t$. Since 
\begin{align*}
b_{(q+1)t+p}-b_{qt+p}& =c_{qt+p} +  \cdots  + c_{(q+1)t}+ c_{(q+1)t+1}+\cdots+ c_{(q+1)t+p-1}\\
 & = c_p + c_{p+1} + \cdots + c_t+ c_1 + \cdots + c_{p-1} =b_{t+1}-b_1,
\end{align*}
so $b_{(q+1)t+p}=b_{qt+p}+ b_{t+1}-b_1$.   By  the  hypothesis of the induction,  we have  $b_{qt+p}= b_p+q(b_{t+1}-b_1)$ and hence
$b_{(q+1)t+p}-b_p=(q+1)(b_{t+1}-b_1)$ as desired.
Clearly, the general case implies the particular one.
\end{proof}

Given integers $2\leq s \leq r$, let $M$ be  the integer given by the theorem \ref{mainpol1}. 
Let $N_0 \#  \Gb_{r,s}^{\B_M}$, and  assume that $\Cc$, $\D$ and  $\Ee$ are periodic.
To prove Theorem \ref{cor1}, 
we  show that $\Gb_{r,s}^{\B_{n}}$ is empty set for each $n\geq M_0$, where
\[M_0:=
\begin{cases}
  t \cdot (N_0 + M +2) & \text{if $t\leq  M$;}\\
 t \cdot (N_0+2) & \text{if $t> M$.}
\end{cases} 
\]
By contrary, we assume that $\Gb_{r,s}^{\B_n} \not = \emptyset$ and  
$f \in \Gb_{r,s}^{\B_n}$ for some $n\geq M_0.$ We  prove that there  exists    $N_0 +1$ pairwise distinct 
polynomials $f_1, \cdots, f_{N_0}, f_{N_0+1}$ in $\Gb_{r,s}^{\B_M}$.
 This  gives a contradiction and complete the proof of  the corollary. 
Let us  assume that   $\Cc$ is periodic with period $t\geq 1$.
 For each $1\leq j \leq N_0+1$,  the following  polynomials 
$$f_j(x):=f(x+j(b_{t+1}-b_1)),$$
are distinct elements of  $\Gb_{r,s}^{\B_M}$. To see this,
 it is enough to check that $f_j(b_i)$  is an $s$-powerful element in $k$, for $1\leq i \leq M$ and $1 \leq j\leq N_0+1$. 
If $t>M$, then by definition of $f_j$'s and part (i) of  above lemma,  we have
$$f_j(b_i)=f(b_i +j (b_{t+1}-b_1)) = f(b_{j t+i}).$$
Since $j t+i \leq (N_0+1)t+M \leq (N_0+2)t\leq n $ and $f\in \Gb_{r,s}^{\B_n}$, so $f(b_{jt+i})$ and hence $f_j(b_i)$
  is an $s$-powerful element in $k$.
In the case $t\leq  M$,  for each $1\leq i \leq M$ we write $i=i_1 t+ i_2$ for some $0\leq i_1 \leq M$ and $  0 \leq i_2\leq t$. Then,
 by part (i) of  above lemma, we have 
\begin{align*}
f_j(b_i) &   =  f(b_i +j(b_{t+1}-b_1))\\
 &   =f(b_{i_2} + i_1(b_{t+1}- b_1)  + j(b_{t+1}-b_1)) \\
&  =  f(b_{i_2} + (i_1+j)(b_{t+1}- b_1))= f(b_{(i_1+j)t+i_2}).
\end{align*}
%
Since $(i_1+j)t+i_2 \leq (N_0+M+1) t + t = (N_0+M+2) t\leq n$ and $f\in \Gb_{r,s}^{\B_n}$, 
so $f(b_{(i_1+j)t+i_2})$ and hence $f_j(b_i)$ is an $s$-powerful element in $k^*$.

In the cases of $\D$ or $\Ee$, one can get   result by  similar arguments. 
Indeed, for  each $f \in \Gb_{r,s}^{\B_n}$  it is enough to consider  respectively the following polynomials
$$f_j(x):=f(u x),  \ \text{with}\ u=(b_{t+1}/b_1)^{(j-1)}, \ \text{for}\ 1\leq j \leq N_0+1,$$ 
$$f_j(x):=v^r f\left( \frac{1}{v}\right),   v=\frac{1}{x}+j \left( \frac{1}{b_{t+1}} -\frac{1}{b_1}\right), \  
\ \text{for}\  1\leq j \leq N_0+1.$$
Note that given any polynomial $f \in k[x]$ of degree $ r \geq 2$ and any $c\in k$, the  function
$$g(x)=w^r f\left( \frac{1}{w}\right),\ \text{with}\ w=\frac{1}{x}+c,$$ is a polynomial of  degree $r$, too.

\section*{Acknowledgment}

This work is part of my Ph.D. thesis at the Universidade Federal do Rio de Janeiro,  Brazil.
 I would like to thank my
supervisor Amilcar Pacheco for his suggestions and comments  during  my Ph.D program.
I would also like to thank Fabian Pazuki for valuable comments on the first version of this paper.


\bibliographystyle{abbrv} 
\bibliography{Refs-pol-pow}{} 

\begin{bibdiv}
\begin{biblist}

\bib{BERCZES}{article}{
      author={B\'{e}rczes, Attila},
      author={Evertse, Jan-Hendrik},
      author={Gy\H{o}ry, K\'{a}lm\'{a}n},
       title={Effective results for hyper- and superelliptic equations over
  number fields},
        date={2013},
        ISSN={0033-3883,2064-2849},
     journal={Publ. Math. Debrecen},
      volume={82},
      number={3-4},
       pages={727\ndash 756},
         url={https://doi.org/10.5486/pmd.2013.5748},
      review={\MR{3066441}},
}

\bib{Bombieri2006}{book}{
      author={Bombieri, Enrico},
      author={Gubler, Walter},
       title={Heights in {D}iophantine geometry},
      series={New Mathematical Monographs},
   publisher={Cambridge University Press, Cambridge},
        date={2006},
      volume={4},
        ISBN={978-0-521-84615-8; 0-521-84615-3},
         url={https://doi.org/10.1017/CBO9780511542879},
      review={\MR{2216774}},
}

\bib{Davenport}{article}{
      author={Davenport, H.},
      author={Lewis, D.~J.},
      author={Schinzel, A.},
       title={Polynomials of certain special types},
        date={1964},
        ISSN={0065-1036},
     journal={Acta Arith.},
      volume={9},
       pages={107\ndash 116},
         url={https://doi.org/10.4064/aa-9-1-107-116},
      review={\MR{163880}},
}

\bib{Hindry2000}{book}{
      author={Hindry, Marc},
      author={Silverman, Joseph~H.},
       title={Diophantine geometry},
      series={Graduate Texts in Mathematics},
   publisher={Springer-Verlag, New York},
        date={2000},
      volume={201},
        ISBN={0-387-98975-7; 0-387-98981-1},
         url={https://doi.org/10.1007/978-1-4612-1210-2},
        note={An introduction},
      review={\MR{1745599}},
}

\bib{Ih2008}{article}{
      author={Ih, Su-ion},
       title={Algebraic points on the projective line},
        date={2008},
        ISSN={0304-9914,2234-3008},
     journal={J. Korean Math. Soc.},
      volume={45},
      number={6},
       pages={1635\ndash 1646},
         url={https://doi.org/10.4134/JKMS.2008.45.6.1635},
      review={\MR{2449921}},
}

\bib{Mahler1964}{article}{
      author={Mahler, K.},
       title={An inequality for the discriminant of a polynomial},
        date={1964},
        ISSN={0026-2285,1945-2365},
     journal={Michigan Math. J.},
      volume={11},
       pages={257\ndash 262},
         url={http://projecteuclid.org/euclid.mmj/1028999140},
      review={\MR{166188}},
}

\bib{Masser2007}{article}{
      author={Masser, David},
      author={Vaaler, Jeffrey~D.},
       title={Counting algebraic numbers with large height. {II}},
        date={2007},
        ISSN={0002-9947,1088-6850},
     journal={Trans. Amer. Math. Soc.},
      volume={359},
      number={1},
       pages={427\ndash 445},
         url={https://doi.org/10.1090/S0002-9947-06-04115-8},
      review={\MR{2247898}},
}

\bib{Neukirch1999}{book}{
      author={Neukirch, J\"{u}rgen},
       title={Algebraic number theory},
      series={Grundlehren der mathematischen Wissenschaften [Fundamental
  Principles of Mathematical Sciences]},
   publisher={Springer-Verlag, Berlin},
        date={1999},
      volume={322},
        ISBN={3-540-65399-6},
         url={https://doi.org/10.1007/978-3-662-03983-0},
        note={Translated from the 1992 German original and with a note by
  Norbert Schappacher, With a foreword by G. Harder},
      review={\MR{1697859}},
}

\bib{Northcott1949}{article}{
      author={Northcott, D.~G.},
       title={An inequality in the theory of arithmetic on algebraic
  varieties},
        date={1949},
        ISSN={0008-1981},
     journal={Proc. Cambridge Philos. Soc.},
      volume={45},
       pages={502\ndash 509},
         url={https://doi.org/10.1017/s0305004100025202},
      review={\MR{33094}},
}

\bib{Pasten2013}{article}{
      author={Pasten, Hector},
       title={Powerful values of polynomials and a conjecture of {V}ojta},
        date={2013},
        ISSN={0022-314X,1096-1658},
     journal={J. Number Theory},
      volume={133},
      number={9},
       pages={2964\ndash 2998},
         url={https://doi.org/10.1016/j.jnt.2013.03.001},
      review={\MR{3057059}},
}

\bib{Riben1974}{article}{
      author={Ribenboim, P.},
       title={Polynomials whose values are powers},
        date={1974},
        ISSN={0075-4102,1435-5345},
     journal={J. Reine Angew. Math.},
      volume={268/269},
       pages={34\ndash 40},
         url={https://doi.org/10.1515/crll.1974.268-269.34},
      review={\MR{364202}},
}

\bib{Schinzel}{article}{
      author={Schinzel, A.},
      author={Tijdeman, R.},
       title={On the equation {$y\sp{m}=P(x)$}},
        date={1976},
        ISSN={0065-1036},
     journal={Acta Arith.},
      volume={31},
      number={2},
       pages={199\ndash 204},
         url={https://doi.org/10.4064/aa-31-2-199-204},
      review={\MR{422150}},
}

\bib{Schmidt1993}{article}{
      author={Schmidt, Wolfgang~M.},
       title={Northcott's theorem on heights. {I}. {A} general estimate},
        date={1993},
        ISSN={0026-9255,1436-5081},
     journal={Monatsh. Math.},
      volume={115},
      number={1-2},
       pages={169\ndash 181},
         url={https://doi.org/10.1007/BF01311215},
      review={\MR{1223249}},
}

\bib{Schmidt1995}{article}{
      author={Schmidt, Wolfgang~M.},
       title={Northcott's theorem on heights. {II}. {T}he quadratic case},
        date={1995},
        ISSN={0065-1036,1730-6264},
     journal={Acta Arith.},
      volume={70},
      number={4},
       pages={343\ndash 375},
         url={https://doi.org/10.4064/aa-70-4-343-375},
      review={\MR{1330740}},
}

\bib{Vojta1987}{book}{
      author={Vojta, Paul},
       title={Diophantine approximations and value distribution theory},
      series={Lecture Notes in Mathematics},
   publisher={Springer-Verlag, Berlin},
        date={1987},
      volume={1239},
        ISBN={3-540-17551-2},
         url={https://doi.org/10.1007/BFb0072989},
      review={\MR{883451}},
}

\bib{Vojta1998}{article}{
      author={Vojta, Paul},
       title={A more general {$abc$} conjecture},
        date={1998},
        ISSN={1073-7928,1687-0247},
     journal={Internat. Math. Res. Notices},
      number={21},
       pages={1103\ndash 1116},
         url={https://doi.org/10.1155/S1073792898000658},
      review={\MR{1663215}},
}

\bib{Vojta2011}{incollection}{
      author={Vojta, Paul},
       title={Diophantine approximation and {N}evanlinna theory},
        date={2011},
   booktitle={Arithmetic geometry},
      series={Lecture Notes in Math.},
      volume={2009},
   publisher={Springer, Berlin},
       pages={111\ndash 224},
         url={https://doi.org/10.1007/978-3-642-15945-9_3},
      review={\MR{2757629}},
}

\bib{Walsh1999}{incollection}{
      author={Walsh, P.~G.},
       title={On a conjecture of {S}chinzel and {T}ijdeman},
        date={1999},
   booktitle={Number theory in progress, {V}ol. 1
  ({Z}akopane-{K}o\'{s}cielisko, 1997)},
   publisher={de Gruyter, Berlin},
       pages={577\ndash 582},
      review={\MR{1689531}},
}

\end{biblist}
\end{bibdiv}

\end{document}